\documentclass{amsart}
\usepackage{amsfonts}
\usepackage{amssymb}
\usepackage{graphicx}
\usepackage{amsmath}

\newtheorem{theorem}{Theorem}

\newtheorem{lemma}[theorem]{Lemma}

\newtheorem{proposition}[theorem]{Proposition}

\begin{document}
\title{ Wreath products by a Leavitt path algebra and affinizations}

\author{Adel Alahmadi}
\address{Department of Mathematics, King Abdulaziz University, P.O.Box 80203, Jeddah, 21589, Saudi Arabia}
\email{adelnife2@yahoo.com}
\author{Hamed Alsulami}
\address{Department of Mathematics, King Abdulaziz University, P.O.Box 80203, Jeddah, 21589, Saudi Arabia}
\email{hhaalsalmi@kau.edu.sa}


\maketitle

\begin{abstract}
We introduce ring theoretic constructions that are similar to the construction of wreath product of groups. In particular, for a given graph  $\Gamma=(V,E)$ and an associate algebra $A,$ we construct an algebra $B=A\, wr\, L(\Gamma)$ with the following property: $B$ has an ideal $I$,which consists of (possibly infinite) matrices over $A$, $B/I\cong L(\Gamma)$, the Leavitt path algebra of the graph $\Gamma$.
\medskip
\par Let $W\subset V$ be a hereditary saturated subset of the set of vertices [1], $\Gamma(W)=(W,E(W,W))$ is the restriction of the graph $\Gamma$ to $W$, $\Gamma/W$ is the quotient graph [1]. Then $L(\Gamma)\cong L(W) wr L(\Gamma/W)$.
\medskip
\newline As an application we use wreath products to construct new examples of (i) affine algebras with non-nil Jacobson radicals, (ii) affine algebras with non-nilpotent locally nilpotent radicals.
\end{abstract}

\vspace{1cm}
\section{Actions by Semigroups}
 ~\par Let $S$ be a semigroup with zero, that is, there exists an element $s_{0}$ such that $s_{0}S=\{s_{0}\}= Ss_{0}$. Suppose that the semigroup $S$ acts on a set $X$ both on the left and on the right, that is, there are mappings $S\times X\longrightarrow X$, $X \times S \longrightarrow X$ such that $s_{1}(s_{2}x)=(s_{1}s_{2})x$,$(xs_{1})s_{2}=x(s_{1}s_{2})$ for arbitrary elements $s_{1},s_{2}\in S$;$x \in X$.
 \medskip
 \par We assume that $X$ is a set with zero, that is, there exists an element $x_{0}$ such that $s x_{0}=x_{0},x_{0}s=x_{0},x s_{0}=s_{0}x=x_{0}$ for arbitrary elements $s \in S,x \in X$.

 Suppose further that the left and right actions of the semigroup $S$ on $X$ have the following properties. For arbitrary elements $s\in S,$ $x\in X$:
 \begin{enumerate}
   \item [(1)] if $s(xs)=x_0$ then $xs=x_0.$ If $s(xs)\neq x_0$ then $s(xs)=x;$
   \item [(2)] if $(sx)s=x_0$ then $sx=x_0.$ If $(sx)s\neq x_0$ then $(sx)s=x.$
 \end{enumerate}

For a field $F$ let $F_{0}[S]$ denote the reduced semigroup algebra, $F_{0}[S]=F[S]/F s_{0}$.
 \medskip
\par Let $A$ be an $F$-algebra. Let $M_{X\times X} (A)$ denote the algebra of possibly infinite $X\times X$ - matrices over $A$ with only finitely many nonzero entries.
For elements $s \in S;x,y \in X;a \in A$ let $a_{x,y}$ denote the matrix having  $a$ in the $x$-th row and $y$-th column and zeros in all other entries.

\bigskip

\par We will define an algebra structure on $F_{0}[S]+ M_{X\times X}(A)$. For arbitrary elements $s \in S;x,y \in X;a \in A$ we define

\[ sa_{x,y} = \left\{
   \begin{array}{l l}
     0, & \quad \text{if $sx=x_{0}$ }\\
     a_{sx,y}, & \quad \text{if $sx \neq x_{0}$ }
   \end{array} \right.\]

\[ a_{x,y}s = \left\{
   \begin{array}{l l}
     0, & \quad \text{if $ys=x_{0}$ }\\
     a_{x,ys}, & \quad \text{if $ys \neq x_{0}.$ }
   \end{array} \right.\]

In particular, $F_{0}[S]a_{x_{0},x}=a_{x,x_{0}}F_{0}[S]=(0)$.
\medskip
\begin{lemma}\label{lem1}
The algebra $F_{0}[S]+ M_{X\times X}(A)$ is associative.
\end{lemma}
\begin{proof}
The only nontrivial case that we need to check is $(a_{x,y} s)b_{z,t}=a_{x,y} (sb_{z,t}),$ where $x,y,z,t\in X,$ $s\in S.$
If the left hand side is not equal to zero then $ys=z\neq x_0.$ By the property $(1)$ $sz=s(ys)=y,$ which implies  associativity.
If the right hand side is not equal to zero, then $y=sz\neq x_0.$ As above by $(2)$ $ys=(sz)s=z,$ which again implies  associativity.
This proves the Lemma.
\end{proof}
\medskip

\section{Wreath Product of Algebras}
\par Now let $\Gamma =(V,E)$ be a row finite directed graph with the set of vertices $V$ and the set of edges $E$. For an edge $e\in E$, let $s(e)$ and $r(e)\in V$
denote its source and range respectively. A vertex $v$ for which $s^{-1}(v)$ is empty is called a sink. A path $p=e_{1}...e_{n}$ in a graph $\Gamma $ is
a sequence of edges $e_{1}...e_{n}$ such that $r(e_{i})=s(e_{i+1}),$ $ i=1,2,...,n-1$. In this case we say that the path $p$ starts at the vertex $s(e_{1})$
and ends at the vertex $r(e_{n})$. We refer to $n$ as the length of the path $p$. Vertices are viewed as paths of length $0$. The Cohn algebra $C(\Gamma)$ is presented by generators $V\bigcup\limits^{.}
E\bigcup\limits^{.}E^{*}$ and relations: $v^{2}=v,\ v \in V;\ vw=wv=0;\ v,\ w \in V,\ v \neq w;$\ $s(e)e=e r(e)=e,\ e \in E;\ e^{*}=e^{*}s(e)=r(e)e^{*},\ e \in E;$\ $e^{*} f=0;\ e,f \in E,\ e\neq f;\ e^{*}e=r(e),\ e \in E$. Clearly, the set $S=\{pq^{*}| p,q \text{ are paths on } \Gamma\}\cup \{0\}$ is a semigroup with zero and $C(\Gamma)$ is a reduced semigroup algebra.

\medskip

 If $X,Y$ are nonempty subsets of the set $V$ then we let $E(X,Y)$ denote the set $\{e\in E \mid\, s(e)\in X,  r(e)\in Y\}.$
\medskip \medskip
\par Let $\mathcal{E}$ be a family of pairwise orthogonal idempotents in $A$. We introduce a set $E(V,\mathcal{E})$ of edges connecting $V$ to idempotents from $\mathcal{E}$ such that for every nonsink vertex $v \in V$ the set of edges set $e \in E (v,\mathcal{E}),s(e)=v$ is finite (possibly empty).
\medskip \medskip
 If $v$ is a sink in $\Gamma$, then we assume that $E(v,\mathcal{E})=\emptyset$.
 Now we extend the graph $\Gamma$ to a graph $\widetilde{\Gamma}(\widetilde V,\widetilde E)$, where $\widetilde V=V \cup \mathcal{E}, \widetilde E = E\cup E (V,\mathcal{E}).$
\medskip \medskip
\par Let $\mathcal{P}$ be the subset of the extended Cohn algebra $C(\widetilde{\Gamma})$, which consists of paths, that start in $\Gamma$ and end in $\mathcal{E}$, and zero, so
$\mathcal{P}=\left( \bigcup\limits_{\text{ p is a path} \atop { \text{ on $\Gamma$}}} pE(r(p),\mathcal{E})\right)\cup \{0\}$.
\medskip
\par The Cohn algebra $C(\Gamma)$ is a subalgebra of the Cohn algebra $C(\widetilde{\Gamma})$.
\medskip
\begin{lemma}\label{lem2}
$C(\Gamma)\mathcal{P}\subseteq \mathcal{P}$.
\end{lemma}
\begin{proof}
We have $C(\Gamma)\mathcal{P}= \text{span} C(\Gamma)pe$, where $p$ is a path in $\Gamma$ and $e\in E(V,\mathcal{E})$ with $r(p)=s(e)$. Since $C(\Gamma)p\subseteq C(\Gamma)$, it is sufficient to show that $pq^*e\in\mathcal{P}$ for arbitrary paths $p, q$ in $\Gamma$, $r(p)=r(q)$, $s(q)=s(e)$. Furthermore, it is sufficient to prove that $q^*e \in \mathcal{P}$. If length of $q \geq 1$, then $q^*e=0$. If $q$ is a vertex, then $q^*e=e$. This proves the Lemma.
\end{proof}
\medskip
\par By Lemma \ref{lem2}, $\mathcal{P}$ can be viewed as a left $S$-module. We will define also a structure of a right $S$-action on $\mathcal{P}$ via $p.s=s^{*}p \in P$.

\medskip

\begin{lemma}\label{lem3}
The left and right actions of the semigroup $S$ on $\mathcal{P}$ satisfy $(1),$ $(2).$
\end{lemma}
\begin{proof}
Let us check  property $(1).$ If $s=0$ or $x=0$ then clearly $xs=0.$ Suppose that $s=pq^*\neq 0, x=p_1\neq0.$
Then $s(xs)=ss^*p_1=pq^*qp^*p_1=pp^*p_1.$ The equality $pp^*p_1=0$ means that the path $p$ is not a beginning of the path $p_1,$ in which case $xs=qp^*p_1=0.$ If $pp^*p_1\neq 0$ then $p$ is a beginning of the path $p_1,$ $p_1=pp_2.$ Now, $s(xs)=pp^*p_1=pp^*pp_2=pp_2=p_1=x.$  Property $(2)$ is checked similarly.
This proves the Lemma.
\end{proof}
\medskip

\par Consider the algebra $C(\Gamma)+M_{\mathcal{P}\times\mathcal{ P}}(A)$ that we have defined in Section I. We extend the range function $r$ by $r(0)=1$. Now consider the subalgebra $C(\Gamma)+I$, where $I$ consists of matrices having all $(p,q)$-entries lie in the $r(p)Ar(q).$
\medskip
\par Clearly, $I$ is an ideal of the algebra $C(\Gamma)+I$. For a nonsink vertex $v\in V(\Gamma),$ consider the element $CK(v)=CK(v)^{'}-CK(v)^{''},$ where
$$CK(v)^{'}=v-\sum\limits_{f\in E(\Gamma)\atop{s(f)=v}}ff^{*},CK(v)^{''}=\sum\limits_{e \in E(v,\mathcal{E})\atop{s(e)=v}}(r(e))_{e,e}.$$

\medskip
\begin{lemma}\label{lem4}
$I\ CK(v)=CK(v)\ I=(0)$, for any nonsink vertex $v \in V(\Gamma)$.
\end{lemma}
\begin{proof}
Let $p,q \in\mathcal{ P}$ and $a \in r(p) A r(q),$ where $v \in V(\Gamma)$ is not a sink. We will show that $a_{p,q}CK(v)=0$. If $q$ is the zero or $s(q)\neq v$, then $a_{p,q}v=0$ as $vq=0; a_{p,q}f=0$ as $f^{*}q=0$ and $a_{p,q}r(e)_{e,e}=0$ as $q\neq e$ (the edge $e$ starts at $v$).
\medskip
\par Now suppose that  $q$ be a nonzero path, $s(q)=v$. Suppose at first that length $(q)=1,$ that is, $q=e$ is an edge connecting $v$ with an idempotent $r(e) \in\mathcal{E}$. Then $a_{p,q}v=a_{p,q}; a_{p,q}f=0$ because $f^*q=0$;
$a_{p,q}r(e)_{e,e}=(ar(e))_{p,e}=a_{p,q}, a_{p,q}r(e^{'})_{e^{'},e^{'}}=0$ for an edge $e^{'}\in E(v,\mathcal{E}), e^{'}\neq e$.
Hence $a_{p,q}CK(v)=0$.\\
\par Now suppose that  length $(q)\geq 2$. Then $q=fq^{'},f \in E (v,V(\Gamma))$. In this case $a_{p,q}v=a_{p,q};\ a_{p,q}ff^{*}=a_{p,q};\ a_{p,q}f^{'}f^{'*}=0$, for an edge $f^{'} \in E(v,V(\Gamma)), f^{'}\neq f$.
Now $a_{p,q}r(e)_{e,e}=0,$ because $q\neq e$ and again $a_{p,q}CK(v)=0$. We proved that $CK(v)\ I=(0)$. Similarly, $I\ CK(v)=(0)$. This proves the Lemma.
\end{proof}

\begin{lemma}\label{lem5}
Let $v_{1},....,v_{m}$ be distinct vertices in $V(\Gamma)$. Let $p_{ik},q_{ik},p^{'}_{is},q^{'}_{it}$ be the paths of length $\geq 1$ in $\Gamma$, $r(p_{ik})=r(q_{ik})=r(p^{'}_{is})=r(q^{'}_{it})=v_{i}$. Assume that for each $i$ all paths $p^{'}_{is}$ are distinct; all paths $q^{'}_{it}$ are distinct and all pairs $(p_{ik},q_{ik})$ are distinct. Then the elements $p_{ik}CK(v_{i})^{'}q^{*}_{ik},\ p'_{is}CK(v_{i})^{'},\ CK(v_{i})^{'}q^{'*}_{it},\ CK(v_{i})^{'}$ in $C(\Gamma)$ are linearly independent.
\end{lemma}

\begin{proof}
Suppose that $\alpha_{ik},\beta_{is},\gamma_{it},\xi_{i} \in F$ and
$$\sum_{i,k}\alpha_{ik}p_{ik}CK(v_{i})^{'}q^{*}_{ik}+\sum_{i,s} \beta_{is}p^{'}_{is}CK(v_{i})^{'}+\sum_{i,t}\gamma_{it}CK(v_{i})^{*'}q_{it}+\sum\xi_{i}CK(v_{i})^{'}=0.$$
We take $S_1=\sum\limits_{i,k}\alpha_{ik}p_{ik}CK(v_{i})^{'}q^{*}_{ik},$ $S_2=\sum\limits_{i,s} \beta_{is}p^{'}_{is}CK(v_{i})^{'},$ $S_3=\sum\limits_{i,t}\gamma_{it}CK(v_{i})^{*'}q_{it},$ and $S_4=\sum\limits_{i}\xi_{i}CK(v_{i})^{'}.$ Since semigroup elements involved in different summands $S_{i}, S_j,\ i\neq j$ are distinct, it follows that $S_1=S_2=S_3=S_4 =0.$

Suppose that not all coefficients $\alpha_{ik}$ are equal to zero. Let $d=\max\{\textrm{length}(p_{ik})+\textrm{length}(q_{ik})\ | \alpha_{ik}\neq 0\}$. Suppose that this maximum is achieved at $(i_{0},k_{0})$. Let $f \in E (v_{i},V(\Gamma))$. Then the summand $\alpha_{i_0k_0} p_{i_0k_0}ff^*q^*_{i_0k_0}$ won't cancel in $S_1.$ Hence all $\alpha_{ik}=0$. Equalities $\beta_{is}=\gamma_{it}=\xi_{i}=0$ are proved similarly.
This proves the Lemma.
\end{proof}

\par Let $J$ be the ideal of $C(\Gamma)+I$ generated by all elements $CK(v)$, where $v$ runs over all nonsink vertices from $V(\Gamma)$.
\medskip

\begin{lemma}\label{lem6}
$J\cap I=(0).$
\end{lemma}

\begin{proof}
It is easy to see that for any edge $g \in E(\Gamma)$ we have $g^{*}CK(v)=CK(v)g=0$. Hence an arbitrary element from the ideal $J$ can be represented as $$x=\sum_{i,k}\alpha_{ik}p_{ik}CK(v_{i})q^{*}_{ik}+\sum_{i,s}\beta_{is}p^{'}_{is}CK(v_{i})+\sum_{i,t}\gamma_{it}CK(v_{i})q^{'*}_{it}+\sum_{i}\xi_{i}CK(v_{i}),$$ where $\alpha_{ik},\beta_{is},\gamma_{it},\xi_{i} \in F;p_{ik},q_{ik},p^{'}_{is},q^{'}_{it}$ are paths on $\Gamma$ of length $\geq 1$. If this element lies in $I$ then
$$\sum_{i,k}\alpha_{ik}p_{ik}CK(v_{i})^{'} q^{*}_{ik}+\sum_{i,s}\beta_{is}p^{'}_{is}CK(v_{i})^{'}+\sum_{i,t}\gamma_{it}CK(v_{i})^{'}q^{'*}_{it}+\sum\xi_{i}CK(v_{i})^{'}=0.$$ By Lemma \ref{lem5} $x=0$. This proves the Lemma.
\end{proof}

\par Following G. Abrams and Z. Mesyan [2],  we may view the Leavitt path algebra $L(\Gamma)$ as the quotient algebra $C(\Gamma)/N$, where $N$ is the ideal of $C(\Gamma)$ generated by $CK(v)'$ for all nonsink $v\in V$. Now let $B=(C(\Gamma)+I)/J$. Clearly the algebra $B$ has an ideal $(I+J)/J\cong \sum\limits_{p,q \in \mathcal{P}}(r(p)A r(q))_{p,q}$ and the quotient of $B$ modulo this ideal is isomorphic to the Leavitt path algebra $L(\Gamma)$. We will denote the algebra $B$ as $A\,wr\, L(\Gamma)$ and call it the wreath product of the algebra $A$ and the Leavitt path algebra $L(\Gamma)$. Remark that the construction $A\,wr\, L(\Gamma)$ depends on the set of idempotents $\mathcal{E}$ and the set of edges $E(V,\mathcal{E})$.

\begin{proposition}\label{prop1}
 If $\Gamma$ is a finite graph and the algebra $A$ is finitely generated then $A\,wr\, L(\Gamma)$ is finitely generated.
\end{proposition}

\par Indeed, if $A=\langle a_1, a_2, \cdots, a_m \rangle$, then the algebra $C(\Gamma)+I$ is generated by $V, E, E^*$ and matrices $(a_i)_{00}, (r(e))_{e,0},(r(e))_{0,e},\, e\in E(V,\mathcal{E}).$

\medskip\medskip
Let us discuss some applications of the wreath product construction to the theory of Leavitt path algebras. A subset $W\subseteq V$ is said to be hereditary if $v \in W$ implies $r(s^{-1}(v))\subseteq W$ [1]. The subset $W$ is said to be saturated if $r(s^{-1}(v))\subseteq W$  implies that $v\in W,$ for every non-sink vertex $v\in V$ [1]. The Leavitt path algebra $L(\Gamma)$ has a natural $\mathbb{Z}$-gradation: $deg(v)=0, deg(e)=1, deg(e^*)=-1$.  A Leavitt path algebra  $L(\Gamma)$ is graded simple if and only if  $\Gamma$ does not contains proper hereditary and saturated subsets (see [7]).

\medskip
\par Let $W$ be a hereditary and saturated subset of $V$. The graph $\Gamma(W)=(W, E(W,W))$ is the restriction of the graph $\Gamma$ . Consider the graph $\Gamma / W$ with the set of vertices $V\setminus W$ and the set of edges $E\setminus E(V,W)$. The set $\mathcal{E}= V\setminus W$ in $L(\Gamma/ W)$ is the set of pairwise orthogonal idempotents. Vertices from $W$ are connected to idempotents $\mathcal{E}$ via the edges from $E(W,V\setminus W)$.
\medskip

\begin{proposition}\label{prop2}
$L(\Gamma)\cong L(\Gamma(W)) \,wr\, L(\Gamma/W).$
\end{proposition}

\begin{proof}
Let $\mathcal{P}$ be the set of pathes on the graph $\Gamma$ such that $s(p)\in V\setminus W,$ $r(p)\in W,$ $r(p)$ is the first vertex on the path $p$ that lies in $W.$
An arbitrary element  $a$ from $L(\Gamma)$ can be uniquely represented as $$a=a'+\sum pa_{pq}q^*+\sum pb_{p}+\sum c_q q^*+d$$
where $a'$ is a linear combination of elements $p_1p^*_2$ such that no vertex on $p_1,\, p_2$ lies in $W ;$ $p,\, q\in P\,;$ $a_{pq}, b_p, c_q, d$ lie in the subalgebra of
$L(\Gamma)$ generated by $W,\, E(W,W).$ Let $A=L(\Gamma(W)),\, \Gamma(W)=(W, E(W,W)),$ $\mathcal{E}=W,$ $E(V, \mathcal{E})=E(V,W).$
Straightforward verification show that the mapping $L(\Gamma)\rightarrow A\, wr\, L(\Gamma/W),$ $$a'+\sum pa_{pq}q^*+\sum pb_{p}+\sum c_q q^*+d\mapsto a'+\sum_{p,q} (a_{pq})_{p,q}+\sum_{p}(b_{p})_{p,0}+\sum_{q} (c_q)_{0,q}+(d)_{00}$$ is an isomorphism of algebras.
\end{proof}

\medskip
From [5,6], it follows that the Leavitt path algebra of a finite graph has polynomial growth if and only if it is an iterated wreath product of disjoint unions of cycles and trees.

\medskip
Following [3], we call a vertex $v$ in a connected graph $\Gamma(V,E)$ a \textit{balloon} over a nonempty subset $W$ of $V$ if (1) $v\notin W,$ (2) there is a loop $C\in E(v,v),$ (3) $E(v,W)\neq\emptyset,$ (4) $E(v,V)=\{C\}\cup E(v,W),$ and (5) $E(V,v)=\{C\}.$ If $V$ contains a vertex $v$ which is a balloon over $V\setminus \{v\}$, then we say the graph $\Gamma$ is a balloon extension. Now let $\Gamma$ be a graph and $\Gamma'$ be a balloon extension. Then $L(\Gamma')\cong L(\Gamma) \,wr\, L(C)$, where $C$ is a loop.

\section{Affinizations of countable dimensional algebras }
By an affinization we mean an embedding of a countable dimensional algebra in an affine ( that is, finitely generated ) algebra with preservation of certain properties.
\bigskip
\par In $1981$ K. Beidar [8] constructed an affine algebra with a non-nil Jacobson radical, answering an old question of S. Amitsur. Beidar's construction was modified and generalized by L. Small ( see [10] ).
\bigskip

The next step was done by J. Bell [9] who constructed affinizations of small Gelfand-Kirillov dimensions.
\medskip
\newline Answering a question of K. Zhevlakov, E. Zelmanov [14] constructed an affine algebra with a non-nilpotent locally nilpotent radical.
\medskip

In this section we show how to construct different  affinizations using wreath products.
\medskip

\begin{figure}
  \includegraphics[width=.1\textwidth]{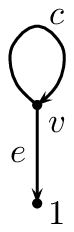}\\
  \caption{}\label{1}
\end{figure}

\par Let $\Gamma(\{v\},\{c\})$ be a loop. The Leavitt path algebra $L(\Gamma)$ is $$\sum_{i\geq 0} Fc^i+\sum_{i\geq 0} F(c^*)^i\cong F[t^{-1},t].$$  Let $A$ be an associative algebra with $1.$
Let $\mathcal{E}=\{1\},$ the only vertex $v$ of the graph $\Gamma$ is connected to $1$ by the edge $e,$ see figure 1 .
Then the set $\mathcal{P}$ from the construction of the wreath product $A\,wr\,L(\Gamma)$ is $\mathcal{P}=\{c^i e, \, i\geq 0\}\cup \{0\},$ $A\,wr\,L(\Gamma)=F[c,c^*]+M_{\mathcal{P}\times\mathcal{P}}(A).$
We will identify the set $\mathcal{P}$ with the set of nonnegative integers $\mathbb{N},$ $0\leftrightarrow 0,$ $c^i e\leftrightarrow i+1,\, i\geq 0.$ Then  $A\,wr\,L(\Gamma)$ can be identified with $F[t^{-1},t]+M_{\mathbb{N}\times \mathbb{N}}(A),$ where $F[t^{-1},t]$ is an isomorphic copy of the vector space of the algebra of Laurent polynomials ( it is not a subalgebra in $A\,wr\,L(\Gamma)$, see [6]);
$t^{-1}t=v,\, tt^{-1}=v-(1)_{1,1}; ta_{i,j}=a_{i+1,j}, t^{-1}a_{i,j}=a_{i-1,j}$ for $i\geq 1,$ $t^{-1}a_{0,j}=0;$ $a_{i,j}t=a_{i,j-1}$ for $j\geq 1,$ $a_{i,0}t=0,$ $a_{i,j}t^{-1}=a_{i,j+1}.$

\bigskip

\par Let $\widetilde{M_{\mathbb{N}\times \mathbb{N}}(A)}$ denote the algebra of infinite $\mathbb{N}\times \mathbb{N}$ matrices over $A$ having only finitely many nonzero entries in each row and in each
column. Clearly, $M_{\mathbb{N}\times \mathbb{N}}(A)\triangleleft \widetilde{M_{\mathbb{N}\times \mathbb{N}}(A)}.$
\par The algebra $L(\Gamma)+\widetilde{M_{\mathbb{N}\times \mathbb{N}}(A)}$ is defined in the same way as the algebra $A\, wr\, L(\Gamma)=L(\Gamma)+M_{\mathbb{N}\times \mathbb{N}}(A).$
Moreover, $M_{\mathbb{N}\times \mathbb{N}}(A)\triangleleft L(\Gamma)+\widetilde{M_{\mathbb{N}\times \mathbb{N}}(A)}.$
\bigskip
\par Suppose that the algebra $A$ is generated by a countable set $a_0,a_1,\cdots.$ Let $a=\sum\limits_{i=0}^\infty (a_i)_{i,i}\in \widetilde{M_{\mathbb{N}\times \mathbb{N}}(A)}.$ Let $B$ denote the affine algebra generated by $t,t^{-1},$ $a,$ $(1)_{0,0}.$ The following assertion is straightforward.

\begin{proposition}\label{prop3}
$F[t^{-1},t]+M_{\mathbb{N}\times \mathbb{N}}(A)\subset B\subset F[t^{-1},t]+\widetilde{M_{\mathbb{N}\times \mathbb{N}}(A)}.$
\end{proposition}

\par Now let $\text{ Rad }(A)$ be the Jacobson or the locally nilpotent radical of the algebra $A$ (see [11] ). Then $M_{\mathbb{N}\times \mathbb{N}}(\text{ Rad }(A))=\text{ Rad }(M_{\mathbb{N}\times \mathbb{N}}(A))\subseteq \text{ Rad }(B).$

\par If $F$ is a countable field then there exists a countably generated commutative domain $A_0,$ which is equal to its Jacobson radical. For example, in the field of rational functions $F(t)$  consider the
subalgebra $A_0=\left\{ \frac{f(t)}{g(t)}\mid  f(0)=0, g(0)=1\right\}.$ It is easy to see that the algebra $A_0$ is countably dimensional and equal to its Jacobson radical. Let $A=F.1+A_0.$
\bigskip
\par The algebra $B$ of Proposition \ref{prop3} is affine. The Jacobson
radical of $B$ contains $M_{\mathbb{N}\times \mathbb{N}}(A)$ and therefore is not nil.
\bigskip
\par Ju. M. Rjabuhin [13] constructed a prime countably generated locally nilpotent algebra $A_0.$ Let $A= F.1+A_0.$
Then the locally nilpotent radical of the affine algebra $B\subset F[t^{-1},t]+\widetilde{M_{\mathbb{N}\times \mathbb{N}}(A)}$ contains $M_{\mathbb{N}\times \mathbb{N}}(A_0)$ and therefore is not nilpotent.

\section*{Acknowledgement}
The authors would  like to express their appreciation to professor  S. K. Jain for carefully reading the manuscript and for offering his comments.
\bigskip
\newline The authors are grateful to the referee for numerous helpful comments.
\bigskip
\newline This paper was funded by King Abdulaziz University, under grant No. (57-130-35-HiCi). The authors, therefore, acknowledge technical and financial support of KAU.


\begin{thebibliography}{9}
\bibitem{AA} Abrams G., Aranda Pino G. \textit{The Leavitt path algebra of a graph.} J. Algebra 293(2): 319-334 (2005).

\bibitem{AM} Abrams G., Mesyan Z., \textit{Simple Lie algebra arising from Leavitt path algebra.} Journal of pure and applied algebra, 216: 2303-2313 (2012).

\bibitem{AAAHA} Alahmadi A., Alsulami H., \textit{Simplicity of Lie algebra of skew elements of Leavitt path algebra.} submitted

\bibitem{AAAHA2} Alahmadi A., Alsulami H., \textit{On the simplicity of Lie algebra of Leavitt path algebra.} submitted

\bibitem{AAJE} Alahmedi A., Alsulami H., Jain S. K., Zelmanov E. \textit{Leavitt path algebras of finite Gelfand-Kirillov dimension.} J. Algebra Appl. 11(6) (2012).

\bibitem{AAJE2} Alahmedi A., Alsulami H., Jain S. K., Zelmanov E. \textit{Structure of Leavitt path algebras of polynomial growth.} PNAS 110(38): 15222-15224 (2013).

\bibitem{AMP} Ara P., Moreno M., Padro E. \textit{Nonstable K-theory for graph algebras .} J. Algebra Represent Theory. (10)157-178 (2007).

\bibitem{BK} Beidar K.I., \textit{Radicals of nitely generated algebras.} Uspekhi Mat. Nauk 36(1981), no. 6 (222), 203 - 204.

\bibitem{BL} Bell, J. \textit{Examples in finite Gelfand-Kirillov dimension.} J. Algebra 263 (2003), no. 1, 159 - 175.

\bibitem{BS} Bell, J.; Small, L. \textit{A question of Kaplansky. Special issue in celebration of Claudio Procesi's 60th birthday.} J. Algebra 258 (2002), no. 1, 386 - 388.

\bibitem{IH} I.N. Herstein, \textit{Noncommutative rings.} Reprint of the 1968 original. With an afterword by L.W. Small. Carus Mathematical Monographs, 15. Mathematical
Association of America, Washington, DC, 1994

\bibitem{KM} Kargapolov M, Merzlyakov Y, \textit{Fundamentals of the theory of groups.} New York, Springer-Verlag (1979).

\bibitem{RJ} Rjabuhin, Ju. M., \textit{A certain class of locally nilpotent rings.} (Russian) Algebra i Logika 7 1968 no. 5, 100–108.

\bibitem{ZE} Zelmanov, E.I., \textit{An example of a finitely generated prime ring.} Sibirsk Mat. Z. 20 (1979), 423 (in Russian); English trans. Siberian Math. J. 20
(1979), 303–304.
\end{thebibliography}
\end{document}